\documentclass[11pt,leqno,oneside]{amsart}
\usepackage{layout}
\usepackage{mathrsfs,dsfont}
\usepackage{amsmath,amstext,amsthm,amssymb,bbm,color}
\usepackage{comment}
\usepackage{charter}
\usepackage{typearea}
\usepackage{pdfsync}
\usepackage[width=6.5in,height=8.7in]{geometry}

\def\bt{\begin{thm}}
\def\et{\end{thm}}
\def\bl{\begin{lem}}
\def\el{\end{lem}}
\def\bd{\begin{defn}}
\def\ed{\end{defn}}
\def\bc{\begin{cor}}
\def\ec{\end{cor}}
\def\bp{\begin{proof}}
\def\ep{\end{proof}}
\def\br{\begin{rem}}
\def\er{\end{rem}}

\newtheorem{thm}{Theorem}[section]
\newtheorem{prop}[thm]{Proposition}
\newtheorem{lem}[thm]{Lemma}
\newtheorem{defn}[thm]{Definition}

\newtheorem{rem}[thm]{Remark}
\newtheorem{cor}[thm]{Corollary}

\numberwithin{equation}{section}

\newcommand{\bthm}{\begin{thm}}
\newcommand{\ethm}{\end{thm}}
\newcommand{\bstp}{\begin{stp}}
\newcommand{\estp}{\end{stp}}
\newcommand{\blemma}{\begin{lemma}}
\newcommand{\elemma}{\end{lemma}}
\newcommand{\bprop}{\begin{prop}}
\newcommand{\eprop}{\end{prop}}
\newcommand{\bpf}{\begin{pf}}
\newcommand{\epf}{\end{pf}}
\newcommand{\bdefn}{\begin{defn}}
\newcommand{\edefn}{\end{defn}}
\newcommand{\brk}{\begin{rmrk}}
\newcommand{\erk}{\end{rmrk}}
\newcommand{\bcrl}{\begin{crl}}
\newcommand{\ecrl}{\end{crl}}

\title[]{Asymptotic normality of Divisors of random Holomorphic Sections on non-compact Complex manifolds}

\address{}
\email{afrimbojnik@sabanciuniv.edu}
\email{ozangunyuz@alumni.sabanciuniv.edu}

\date{}


\begin{document}

\author{Afrim Bojnik}
\author{Ozan Günyüz}
\address{Faculty of Engineering and Natural Sciences, Sabanc{\i} University, \.{I}stanbul, Turkey}

\begin{abstract}
We prove a central limit theorem for smooth linear statistics related to the zero divisors of Gaussian i.i.d. centered holomorphic sections of tensor powers of a Hermitian holomorphic line bundle over a non-compact Hermitian manifold.
  \end{abstract}
\maketitle

%%%%%%%%%%%%%%%%%%%%%%%%%%%%%%%%%%%%%%%%%%%%%%%%%%%%%%%%%%%%%%%%%%%%%%%%%%%%%%%

\section{Introduction and Background}\label{S2}

The statistical aspects connected with zeros of random polynomials and other analytic functions, under different assumptions related to random coefficients, have been the subject of extensive research in mathematics and physics. Among the very important objects were particularly polynomials. Many results regarding both Gaussian and non-Gaussian cases and historical advancements of this polynomial theory may be found, for example, in \cite{BL15, BloomS, BL05, BloomD, ROJ, SHSM, HN08} (and references therein). For instance, in the reference \cite{BL15}, the authors study the complex random variables that are bounded by distribution functions over the entire complex plane $\mathbb{C}$. When we move beyond a large disk of radius $r$, the integral with respect to the two-dimensional Lebesgue measure has an upper bound, which is a function of $r$. This situation is commonly referred to as the tail-end estimate. Before these more recent developments, initial explorations into the distribution of roots of random algebraic equations in a single real variable were pioneered by the likes of Polya and Bloch, Littlewood-Offord, Kac, Hammersley, and Erd\"{o}s-Turan. For readers who are interested in further study, these early efforts can be found in the articles \cite{BlP, Kac43, LO43, HAM56, ET50}.

Another intriguing aspect to consider is the growing body of physics research focused on the equidistribution and probabilistic issues concerning the zeros of complex random polynomials. For initial investigations in this field, sources such as [FH, Hann, NV98] can be referred to.

In their groundbreaking work \cite{SZ99}, which served as a pioneering and initial contribution to the field, Shiffman and Zelditch achieved to prove that the zeros of random sections of the powers $L^{\otimes n}$ of a positive line bundle $L$ over a compact K\"{a}hler  manifold, tend towards a uniform distribution as $n$ goes off to infinity with regard to a natural measure that comes from $L$.

In the broadest context seen so far, various randomness problems emanating from the holomorphic sections of line bundles have been examined in different probabilistic settings (including both Gaussian and non-Gaussian types) as reported in \cite{BCM, BG, CM1, CMM, DLM1, DLM2, Bay16, Shif, SZ99, SZ08, SZ10}.

Our main focus in this article will be the central limit theorem in a general complex geometric setting.

\subsection{Theorem of Sodin and Tsirelson}
Given a sequence $\{\gamma_{j}\}_{j=1}^{\infty}$ of complex-valued measurable functions on a measure space ($G, \sigma$) such that \begin{equation}\sum_{j=1}^{\infty}{|\gamma_j(x)|^{2}}=1\,\,\text{for}\,\,\text{any}\,\,x\in G, \end{equation} following \cite{ST} (and also \cite{SZ10}), a \textit{normalized complex Gaussian process} is defined to be a complex-valued random function $\alpha(x)$ on a measure space $(G, \sigma)$ in the following form \begin{equation}\label{gpro}
                                                      \alpha(x)=\sum_{j=1}^{\infty}{b_{j} \gamma_j(x)},
                                                    \end{equation}where the coefficients $b_{j}$ are i.i.d. centered complex Gaussian random variables with variance one.
The covariance function of $\alpha(x)$ is defined by \begin{equation}\label{cov}
                                                       \beta(x, y)= \mathbb{E}[\alpha(x) \overline{\alpha(y)}]= \sum_{j=1}^{\infty}{\gamma_{j}(x) \overline{\gamma_{j}(y)}}. \end{equation}

A simple observation gives that $|\beta(x, y)| \leq 1$ and $\beta(x, x)=1$.

Let us take a sequence $\{\alpha_{j}\}_{j=1}^{\infty}$ of normalized complex Gaussian processes on a finite measure space $(G, \sigma)$. Let $\lambda(\rho)\in L^2(\mathbb{R}^{+}, e^{\frac{-\rho^2}{2}}\rho d\rho)$. Suppose $\phi: G \rightarrow \mathbb{R}$ is a bounded and measurable function. We consider the following  non-linear functionals (which are also random variables) \begin{equation} \label{nonlin} Z^{\phi}_{n}(\alpha_{n})=\int_{G}{\lambda(|\alpha_{n}(x)|)\phi(x)d\sigma(x)}. \end{equation}

The next theorem (Theorem 2.2 of \cite{ST}) was proved by Sodin and Tsirelson.

\begin{thm}\label{sots}
For each $n=1, 2, \ldots$, let $\beta_{n}(r, s)$ be the covariance functions for the complex Gaussian processes. Assume that the two conditions below hold for all $\nu \in \mathbb{N}$:

\begin{itemize}
 \item[(i)] $$\liminf_{n\rightarrow \infty}{\frac{\int_{G}\int_{G}{|\beta_{n}(r,s)|^{2\nu}\phi(r)\phi(s)d\sigma(r)d\sigma(s)}}{\sup_{r\in G} \int_{G}{|\beta_{n}(r, s)|d\sigma(s)}}}>0.$$
  \item[(ii)] $$\lim_{n\rightarrow \infty}{\sup_{r\in G} \int_{G}{|\beta_{n}(r, s)|d\sigma(s)}}=0.$$
\end{itemize}

Then the distributions of the random variables \begin{equation*}
                                                 \frac{Z^{\phi}_{n}(\alpha_{n})- \mathbb{E}[Z^{\phi}_{n}(\alpha_{n})]}{\sqrt{\mathrm{Var}[Z^{\phi}_{n}(\alpha_{n})]}}
                                               \end{equation*} converge weakly to the normal distribution $\mathcal{N}(0, 1)$ as $n \rightarrow \infty$. If $\lambda$ is increasing, then it is sufficient for $(i)$ to hold only for $\nu=1$.
\end{thm}

Their proof relies on the well-known method of moments from probability theory combined with the so-called diagram technique in order to evaluate the moments of the non-linear functionals and compare them with the moments of standard Gaussian distribution. This is a classical approach for proving the central limit theorem for nonlinear functionals of Gaussian fields.

Based on our present knowledge, there are two important generalizations of Theorem \ref{sots} for different frameworks. One of them is due to Zelditch and Shiffman in the prequantum line bundle setting, that is, for random holomorphic sections of an Hermitian line bundle over a compact K\"{a}hler manifold where the first Chern form is equal to the K\"{a}hler form multiplied by the imaginary unit and a suitable normalization constant. Another one was obtained in \cite{Bay16}. In his paper, Bayraktar made an extension of Theorem \ref{sots} to the general random polynomials on $\mathbb{C}^m$ within the weighted pluripotential theory.  In both cases, the off-diagonal and near-diagonal asymptotics of the normalized Bergman kernel which will be discussed below (regarded as covariance functions of normalized complex Gaussian processes defined above) play a crucial role. In a different direction, Nazarov and Sodin (\cite{NS}) investigated the asymptotic normality of linear statistics of zeros of Gaussian entire functions on $\mathbb{C}$ in a more general way by considering the measurable bounded test functions and the clustering property of $k$-point correlation functions. The authors of the current paper have recently proved a central limit theorem for smooth linear statistics associated with the zero divisors of random holomorphic sections in a sequence of positive line bundles over a compact Kähler manifold, see \cite{BG}.

The objective of the current paper is to obtain Theorem \ref{sots} for the centered Gaussian  i.i.d. holomorphic sections of a positive Hermitian line bundle on a non-compact Hermitian manifold, which is more general than the results mentioned in the previous paragraph. For this purpose, we initially present our geometric and probabilistic setup, which largely follows \cite{DLM1}, \cite{DLM2} and \cite{DMS}. As in the aforementioned articles, the asymptotics of the normalized Bergman kernel will be one of the key apparatuses. As far as we know, \cite{DMS} is the first paper that focuses on the non-compact (Hermitian) manifolds and in that paper, the authors mostly study the equidistribution problem concerned with holomorphic sections of Hermitian holomorphic line bundles over these non-compact manifolds.

\subsection{Complex differential geometric setting}We consider a complex connected Hermitian manifold of dimension $m$ denoted by $(X, J, \Theta)$. Here, $J$ represents the complex structure, which is a linear transformation on each tangent space that acts as a multiplication by $\sqrt{-1}$. $\Theta$ is the Hermitian form which is a $(1, 1)$-form associated with the induced Riemannian metric $g^{TX}$ on the Hermitian manifold $X$, which is compatible with the complex structure $J$, that is to say, $$\Theta(u, u')=g^{TX}(Ju, u'), \,\,\,\,g^{TX}(Ju, Ju')=g^{TX}(u, u')\,\,\,\text{for}\,\,\text{all}\,\,u,\, u'\in T_{x}X.$$ Let $(L, h^{L})$ be a Hermitian holomorphic line bundle over $X$, where $h^{L}$ is a smooth metric on $L$. $R^{L}$ will denote the Chern curvature form of $L$ and the first Chern form of $L$ is given by $$c_{1}(L, h^{L})=\frac{\sqrt{-1}}{2\pi} R^{L}.$$ Instead of $h^{L}$, we will use shortly just $h$. We make the assumption that the Riemannian metric $g^{TX}$ is complete. We will write the $n^{th}$ tensor power of $L$ as $L^{n}$. On $L^{n}$, the metric $h$ induces a product metric, which we write as $h^{\otimes n}:= h^n$. $H^{0}(X, L)$ is the vector space of all holomorphic sections of $L$ on $X$.

The Hermitian form $\Theta$ determines a natural volume form $dv_{X}$ of the Riemannian metric $g^{TX}$ defined by $dv_{X}=\frac{1}{m!}\Theta^{m}$. The space of compactly supported  smooth sections of $L$ is denoted by $\mathcal{C}^{\infty}(X, L)$.  We define the following $L^2$-type inner product using the metrics $g^{TX}$ and $h^{L}$, for any $s, s' \in \mathcal{C}^{\infty}(X, L)$, \begin{equation}\label{ipn}
                                                                                                       \langle s, s' \rangle_{L^2(X, L)}=\int_{X}{\langle s(x), s'(x) \rangle_{h^{L}}dv_{X}(x)}
                                                                                                      \end{equation}
The completion of $\mathcal{C}^{\infty}(X, L)$ with respect to the inner product (\ref{ipn}) will be denoted by $L^{2}(X, L)$ consisting of square integrable smooth sections of $L$. We take into consideration the space of square integrable holomorphic sections of $L$, namely,  $H^{0}_{(2)}(X, L)=H^{0}(X, L) \cap L^{2}(X, L)$. $H^{0}_{(2)}(X, L)$ is a closed subspace of $L^{2}(X, L)$. Indeed, from Corollary 2.4 of \cite{La}, we have for any compact set $U\subset X$ there exists a constant $B_{U}>0$ such that \begin{equation}\label{CE} \sup_{x\in U}{|s(x)|_{h}} \leq B_{U}\|s\|_{L^{2}(X, L)}\end{equation} for all $s\in H^{0}_{(2)}(X, L)$, which implies also that $H^{0}_{(2)}(X, L)$ is closed. As a closed subspace of a separable Hilbert space, it is a separable Hilbert space as well.

Let $K_{n}: L^{2}(X, L^{n})\rightarrow H^{0}_{(2)}(X, L^{n})$ be the orthogonal projection onto its closed and separable subspace $H^{0}_{(2)}(X, L^{n}).$ The reproducing kernel of this projection exists due to the relation (\ref{CE}) above. It is called the Bergman kernel, denoted by $K_{n}(x, y)$, for details see, e.g., \cite{MM1}, section $1.4$. If $d_{n}=\dim{H^{0}_{(2)}(X, L^{n}})$ ($1\leq d_{n} \leq \infty$) and $\{S^{n}_{j} \}_{j=1}^{d_{n}}$ is an orthonormal basis for $H^{0}(X, L^{n}),$ then by using the orthonormal representation of $K_{n}(x, y)$ with respect to the basis $\{S^{n}_{j} \}_{j=1}^{d_{n}}$ and the reproducing kernel property for the space $H^{0}_{(2)}(X, L^{n})$, we get $$K_{n}(x, y)=\sum_{j=1}^{d_{n}}{S^{n}_{j}(x)\otimes S^{n}_{j}(y)^{*}}\in L^{n}_{x} \otimes (L^{n}_{y})^{*},$$where $S^{n}_{j}(y)^{*}=\langle \, .\,, S^{n}_{j}(y)\rangle_{h^{n}}\in (L^{n}_{y})^{*}$. This series is locally uniformly convergent with its derivatives of all order. We write \begin{equation}\label{bergf} K_{n}(x, x)=\sum_{j=1}^{d_{n}}{|S^{n}_{j}(x)|_{h}^{2}},\end{equation} called the Bergman kernel function. We see that $d_{n}=\int_{X}{K_{n}(x, x)dv_{X}(x)}$(which can possibly be infinite). Thus, the Bergman kernel function has the dimensional density for the space $H^{0}_{(2)}(X, L^{n})$.

We also assume that there are positive constants $B, C$ and $\epsilon$ such that the following three conditions on $(X, J, \Theta)$ and $(L, h^{L})$, \begin{itemize}
                                                                                 \item[(a)] $\sqrt{-1} R^{L} \geq \epsilon \Theta,$
                                                                                 \item[(b)] $\sqrt{-1} R^{\det} \geq -C \Theta,$
                                                                                 \item[(c)] $|\partial \Theta|_{g^{TX}} \leq B$,
                                                                               \end{itemize}where $R^{\text{det}}$ is the curvature of the holomorphic connection $\nabla^{\text{det}}$ on the dual canonical bundle (or anticanonical bundle)$K^{*}_{X}=\det{(T^{(1, 0)}(X))}$.

These three conditions together imply the spectral gap property of Kodaira-Laplace operator, which gives the asymptotic expansion of Bergman kernel on non-compact setting, for a thorough examination, see Chapter 6 of \cite{MM1}(also \cite{DMS}). Also note, in particular here, that if $X$ is a K\"{a}hler manifold, then the condition (c) is satisfied automatically since $\partial \Theta=0$.

\subsection{Random holomorphic sections and divisors} Let $s\in H^{0}(X, L)\backslash \{0\}$. We denote by $Z_s$ the set of zeros of $s$. $Z_s$ is a purely $1$-codimensional analytic variety of $X$. The \textit{divisor} set $\text{Div}(s)$ is defined by \begin{equation}\label{div}
\text{Div}(s)= \sum_{U \subset Z_s}{\text{ord}^{s}_{U}U}.
\end{equation}
Here $U$ ranges over all irreducible analytic hypersurfaces within $Z_{s}$. $\text{ord}^{s}_{U}\in \mathbb{N} \backslash \{0\}$ is the vanishing order of $s$ on $U$. Let $U\subset X$ be any analytic hypersurface, by the symbol $[U]$, we mean the current of integration on $U$, defined by $$\varphi \mapsto \int_{U}{\varphi},$$here $\varphi\in \mathcal{D}_{0}^{m-1, m\emph{}-1}(X)$, which represents the space of test forms of type $(m-1, m-1)$. In this paper, we only use real-valued test forms.

The current of integration on $[\text{Div}(s)]$ is defined as follows: \begin{equation}\label{coi}
                                                                         [\text{Div}(s)]= \sum_{U\subset Z_{s}}{\text{ord}^s_{U}} [U].
                                                                       \end{equation}

A complex random variable $W$ is said to be standard Gaussian in case $W= X + \sqrt{-1} Y$, where $X$ and $Y$ are  real centered i.i.d. Gaussian random variables with variance $1/2$.  Let now $\xi=\{\xi_{j}\}_{j=1}^{d_{n}}$ be a sequence of i.i.d. centered  real or complex Gaussian random variables whose probability distribution law is $\mathbb{P}$. Given an orthonormal basis $S=\{S_{j}\}_{j=1}^{d_{n}}$ of $H^0_{(2)}(X, L)$,  we define a random holomorphic section of $L$ by \begin{equation}\label{rhs} \psi^{S}_{\xi}=\sum_{j=1}^{d_{n}}{\xi_{j}S_{j}}.\end{equation}For any test form $\varphi \in \mathcal{D}_{0}^{m-1, m\emph{}-1}(X)$, the random variable $\langle [\text{Div}(\psi^{S}_{\xi}))], \varphi \rangle$ is called \textit{smooth linear statistic}. If this random variable is integrable for any test form $\varphi$, then the map, called the expectation of $[\text{Div}(\psi^{S}_{\xi}))]$,

 \begin{equation}\label{Expe}
                                                        \varphi \mapsto \mathbb{E}[\langle [\text{Div}(\psi^{S}_{\xi})], \varphi  \rangle]
                                                      \end{equation} defines a $(1, 1)$-current on $X$. The variance is defined as \begin{align}\label{expvari}
  \mathrm{Var} \langle [\text{Div}(\psi^{S}_{\xi})], \, \varphi \rangle = \mathbb{E}\langle [\text{Div}(\psi^{S}_{\xi})], \, \varphi \rangle^{2}- (\mathbb{E}\langle [\text{Div}(\psi^{S}_{\xi})], \, \varphi \rangle)^{2}.\end{align}

Sections as in (\ref{rhs}), with probability one, are holomorphic and the distribution of these standard Gaussian random holomorphic sections do not depend on the orthonormal basis chosen, we refer the reader to \cite{DLM2}(Proposition 2.1 and Proposition 2.3 there) for more comprehensive information.

The Poincar\'{e}-Lelong formula is one indispensable instrument that will be of use in the proof of our main theorem. Let us state it for a divisor of a holomorphic section $s\in H^{0}_{(2)}(X, L^{n})$. \begin{equation}\label{poinle}
  [\text{Div}(s)]= dd^c \log{|s|^{2}_{h^{n}}} + nc_{1}(L, h^{L}).
\end{equation}Here and throughout $dd^c= \frac{\sqrt{-1}}{2\pi} \partial \overline{\partial}$. $ [\text{Div}(s)]$ is a positive closed $(1, 1)$-current. $\mathbb{E}[[\text{Div}(\psi^{S}_{\xi})]]$ is also a positive $(1, 1)$-current by Theorem 1.1 in \cite{DLM2}.

In \cite{DLM2}, the authors show that, almost surely under the standard Gaussian probability measure, the random holomorphic sections described in (\ref{rhs}) are not $L^{2}$-integrable. Furthermore, in the same study, through geometric quantization, by associating these sections with the eigenvalues of a Berezin-Toeplitz operator defined by a sufficiently regular symbol function, they consider weighted versions of (\ref{rhs}), which are $L^{2}$-integrable and form a separable Hilbert subspace of $H^{0}_{(2)}(X, L^{n})$. In their very recent paper \cite{DLM3}, they derive a central limit theorem for zero divisors of these random $L^{2}$-integrable holomorphic sections on the support of this symbol function. Additional details can be found in Section 6 of \cite{DLM3}.

\section{Main Result}

This section contains the formulation and proof of our primary theorem. The important tool that we will employ, as described in Section 1, will be the near-diagonal and off-diagonal Bergman kernel asymptotics in the current setting, which is a theorem proven in \cite{DLM1}. To develop our approach, we will make use of three essential geometric assumptions (a), (b), and (c) given in Subsection 1.1.

Initially, the necessary transformations between coordinates will be introduced by utilizing both geodesic coordinates and Riemannian distance.

Let $F^{L}\in \text{End}(T^{(1, 0)}X)$ be such that $x \in X$ and  for arbitrary $u, u' \in  T_{x}^{(1, 0)}X$, \begin{equation}\label{endo}
                                                                                                     R_{x}^{L}(u, u')=g_{x}^{TX}(F^{L}u, u').
                                                                                                   \end{equation}

Fix $x\in X$ and pick an orthonormal basis $\{p_{j}\}_{j=1}^{m}$ of $(T_{x}^{(1, 0)}X, g_{x}^{TX})$ such that \begin{equation}\label{endo1}
                                                                                                  F_{x}^{L} p_{j}= \kappa_{j}(x) p_{j}.
                                                                                                \end{equation}
Here $\kappa_{j}(x),\,j=1, 2, \ldots, m$ are the eigenvalues of $F^{L}$. The first assumption (a) in Subsection 1.1 gives that $\kappa_{j}(x) \geq \epsilon$ for $j=1, 2, \ldots, m$. Write now \begin{equation}\label{ee}
                                           q_{2j-1}=\frac{1}{\sqrt{2}}(p_{j}+ \overline{p_{j}}),\,\,\,q_{2j}=\frac{\sqrt{-1}}{\sqrt{2}}(p_{j}- \overline{p_{j}}),\,\,j=1, 2, \ldots, m.
                                         \end{equation}These vectors lead to an orthonormal basis of the real tangent space $(T_{x}X, g_{x}^{TX})$, so given that $w=\sum_{j=1}^{2m}{c_{j}q_{j}}\in T_{x}X$, one can have the representation \begin{equation}\label{rep1}
                                                          w=\sum_{j=1}^{m}{\frac{1}{\sqrt{2}}(c_{2j-1}+\sqrt{-1}c_{2j})p_{j}}+ \sum_{j=1}^{m}{\frac{1}{\sqrt{2}}(c_{2j-1}-\sqrt{-1}c_{2j})\overline{p_{j}}}.
                                                        \end{equation} We write $z=(z_1, z_2, \ldots, z_{m})$ with $z_{j}=c_{2j-1}+ \sqrt{-1}c_{2j},\,\,j=1, 2, \ldots, m$. The new coordinate $z$ is referred to as the complex coordinate corresponding to the element $w\in T_{x}X$. From (\ref{rep1}), we have \begin{equation}\label{efk} \frac{\partial}{\partial z_{j}}=\frac{1}{\sqrt{2}}p_{j},\,\,\,\frac{\partial}{\partial \bar{z_{j}}}=\frac{1}{\sqrt{2}}\bar{p_{j}},\end{equation} so (\ref{rep1}) takes the following form $$w=\sum_{j=1}^{m}{\big(z_{j} \frac{\partial}{\partial z_{j}}+ \bar{z_{j}} \frac{\partial}{\partial \bar{z_{j}}}\big)}.$$

Since $X$ is a complete Hermitian manifold, by the well-known theorem of Hopf-Rinow, $X$ is also geodesically complete, which means that, for every $x\in X$, the exponential map $\exp_{x}$ is defined on the whole tangent space $T_{x}X$, so take $u, u' \in T_{x}X$ and, determined as above, let $z$ and $z'$ be the related complex coordinates of $u$ and $u'$, respectively. Consider the following weighted distance function defined in \cite{DLM1}  \begin{equation}\label{wedi}
  \Psi_{x}^{TX}(u, u')^{2}= \sum_{j=1}^{m}{\kappa_{j}(x) |z_{j}- z'_{j}|^{2}}.
\end{equation}

Now we provide the coordinate transformations. For a given $\epsilon_{0}>0$, by using the exponential map, between the small open ball $B(x, 2\epsilon_{0})$ in $X$ with the small open ball $\mathfrak{B}(0, 2 \epsilon_{0})$ in $T_{x}(X)$, let $y=\exp_{x}(u')$ and $x=\exp_{x}(0)$. Let $\text{dist}(. , .)$ be the Riemannian distance of $(X, g^{TX})$. There is some constant $D_{1}>0$ such that for $u, u' \in \mathfrak{B}(0, 2 \epsilon_{0})$, we get \begin{equation}\label{dist1}
                                           D_{1} \text{dist}(\exp_{x}(u), \exp_{x}(u')) \geq \Psi_{x}^{TX}(u, u') \geq \frac{1}{D_{1}} \text{dist}(\exp_{x}(u), \exp_{x}(u')).
                                         \end{equation}

Since $\kappa_{j}(x) \geq \epsilon$, by writing $u=0 \in T_{x}X$ and $u'=u$ in (\ref{wedi}), we have \begin{equation}\Psi_{x}^{TX}(0, u) \geq \epsilon^{\frac{1}{2}} \text{dist}(\exp_{x}(0)=x, \exp_{x}(u)=y). \end{equation} Additionally, when we take into account a compact set $K\subset X$, the positive constants $D_{1}$ and $\epsilon_{0}$ can be selected in such a way that the above inequalities hold uniformly for all $x\in K$.

The next result obtained in \cite{DLM1} (Theorem 5.1) demonstrates the asymptotic properties of the normalized Bergman kernel, including both the near-diagonal and off-diagonal behavior. This result is crucial in establishing the proof of the main theorem. Before this result, let us give the definition of the normalized Bergman kernel. For any $n\in \mathbb{N}$ and $x, y \in X$, \begin{equation}\label{nbk}
                                                                                                   \Gamma_{n}(x, y)=\frac{|K_{n}(x, y)|_{h^{n}_{x} \otimes (h^{n}_{y})^{*}}}{K_{n}(x, x)^{1/2} K_{n}(y, y)^{1/2}}.
                                                                                                 \end{equation}
\begin{thm} \label{asym}
Suppose that a relatively compact set $V \subset X$ is given. We have the following uniform estimates for the normalized Bergman kernel for any $x, y \in V$: Fix $k\geq 1$ and $b_{0}> \sqrt{\frac{16k}{\epsilon}}$ such that for large enough $n$ (with $b_{0}\sqrt{\frac{\log{n}}{n}} \leq 2\epsilon_{0}$), one has \begin{equation*}\label{bka} \Gamma_{n}(x, y)= \begin{cases}
                                                   (1+ o(1))\exp{\big( -\frac{n}{4}\Psi_{x}^{TX}(0, u)^{2}\big)}, & \text{if uniformly for}\,\,  \text{dist}(x, \exp_{x}(u)=y)=|u| \leq b_{0}\sqrt{\frac{\log{n}}{n}}, \\
                                                   O(n^{-k}), & \text{if uniformly for}\,\, \text{dist}(x, \exp_{x}(u)=y)=|u| \geq  b_{0}\sqrt{\frac{\log{n}}{n}}.\end{cases} \end{equation*}

\end{thm}

Now we are ready to state and prove our main theorem, which was also suggested in Subsection 3.3 of \cite{DLM2}. For the proof, we use the arguments from \cite{SZ10}.
\begin{thm}\label{maint}
Let $(L, h)$ be a positive Hermitian holomorphic line bundle over a non-compact complete  Hermitian manifold $(X, \Theta)$ of dimension $m$ satisfying the assumptions (a), (b) and (c) given in Subsection 1.1. Suppose that $H^{0}_{(2)}(X, L^{n})$ is endowed with the standard Gaussian measure.  Let $s_{n}\in H^{0}_{(2)}(X, L^{n})$ and $\phi$ be a real $(m-1, m-1)$-form on $X$ with $\mathcal{C}^{3}$-coefficients and $dd^{c} \phi \not\equiv 0$. Then the distributions of the random variables \begin{equation}\label{clte}
                                                                       \frac{\langle [\textnormal{Div}(s_{n})], \phi \rangle - \mathbb{E}[\langle [\textnormal{Div}(s_{n})], \phi \rangle]}{\sqrt{\textnormal{Var}[\langle [\textnormal{Div}(s_{n})], \phi \rangle]}}
                                                                     \end{equation} weakly converges towards the standard (real) Gaussian distribution $\mathcal{N}(0, 1)$ as $n\rightarrow \infty$.

\end{thm}

\begin{proof}
We use the machinery that we gave in the introduction. The normalized Gaussian processes $\alpha_{n}$ on $X$ will be constructed as follows. Take a measurable section $e_{L}$ of $L$ such that $e_{L}: X \rightarrow L$ with $|e_{L}(x)|_{h}=1$ for any $x\in X$. Take an orthonormal basis $\{s^{j}_{n}\}^{d_{n}}_{j=1}$ of $H^{0}_{(2)}(X, L^{n})$ for each $n\in \mathbb{N}$, where  $s^{j}_{n}= \varphi^{j}_{n} e^{n}_{L}$ and $d_{n}\in \mathbb{N}_{\geq 1} \cup \{\infty\}$. Write \begin{equation}\label{pro}
                                                                f^{j}_{n}(x)=\frac{\varphi^{j}_{n}(x)}{\sqrt{K_{n}(x, x)}},\,\,j=1, 2, \ldots, d_{n}.
                                                              \end{equation} Observe that $|\varphi^{j}_{n}|=|s^{j}_{n}|_{h^{n}}$ and $\sum_{j=1}^{d_{n}}{|f^{j}_{n}(x)|^{2}}=1$ by the relation (\ref{bergf}). Therefore, we can define a normalized complex Gaussian process on $X$ for each $n\in \mathbb{N}$ as follows: \begin{equation}\label{cgp}
                                                                         \alpha_{n}=\sum_{j=1}^{d_{n}}{b_{j}f^{j}_{n}},
                                                                       \end{equation}here the coefficients $b_{j}$ are i.i.d. complex Gaussian centered random variables with variance one. We see that a random holomorphic section $s_{n}=\sum_{j=1}^{d_{n}}{b_{j}s_{n}^{j}}$ can be written as \begin{equation}\label{sect}
                                                                        s_{n}=\sum_{j=1}^{d_{n}}{b_{j}s_{n}^{j}}=\sqrt{K_{n}(x, x)} \alpha_{n} e^{n}_{L},
                                                                     \end{equation} where the normalized complex Gaussian process appears.  Observe also from (\ref{sect}) that \begin{equation}\label{norm}
              |\alpha_{n}(x)|=\frac{|s_{n}(x)|_{h^{n}}}{\sqrt{K_{n}(x, x)}}.
            \end{equation}Now we calculate the covariance functions $\beta_{n}$ of the complex Gaussian processes $\alpha_{n}$. Note that, from the fact that the complex Gaussian random coefficients $b_{j}$ in (\ref{cgp}) are centered, i.i.d., and have variance one, it follows \begin{equation}\label{last}\text{Var}[b_{j}]=\mathbb{E}[|b_{j}|^{2}]=1,\,\,\,\mathbb{E}[b_{k}\bar{b_{l}}]=0\,\,\,\text{if}\,\,k\neq l.\end{equation} By the relation (\ref{cov}), (\ref{cgp}) and (\ref{last}), we have \begin{equation}\label{cov1}
                      \beta_{n}(x, y)=\mathbb{E}[\sum_{j=1}^{d_{n}}{b_{j}f^{j}_{n}(x)}\,  \overline{\sum_{j=1}^{d_{n}}{b_{j}f^{j}_{n}(y)}}]=\sum_{j=1}^{d_{n}}{f^{j}_{n}(x) \overline{f^j_{n}(y)}}.
                    \end{equation}

On the other hand, after some calculations, we get \begin{equation}\label{normb}
                                                    |K_{n}(x, y)|_{h^{n}_{x} \otimes (h^{*}_{y})^{n}}= \sqrt{\sum_{j=1}^{d_{n}}{|\varphi^{j}_{n}(x)|^{2} |\varphi^{j}_{n}(y)|^{2}} + \sum_{\substack{j=1\\
                                                    j < k}}^{d_{n}}{2 \Re{(\varphi^{j}_{n}(x) \,\overline{\varphi^{j}_{n}(y)}\,\overline{\varphi^{k}_{n}(x)} \, \varphi^{k}_{n}(y)}})}
                                                  \end{equation}

(\ref{cov1}), (\ref{normb}) and (\ref{pro}) together give \begin{equation}\label{coveb}\Gamma_{n}(x, y)= |\beta_{n}(x, y)|.\end{equation}

Take $\lambda(\rho)=\log{\rho}$. Let $dv_{X}=\frac{1}{m!}\Theta^{m}$ be the volume form on $X$. Let us fix a real $(m-1, m-1)$ test form $\phi$ with $\mathcal{C}^{3}$ coefficients. Then \begin{equation}\label{testf}
                                                                                       dd^c \phi=\frac{\sqrt{-1}}{2\pi}\partial \bar{\partial} \phi=\psi dv_{X},
                                                                                     \end{equation}The function $\psi$ is a $\mathcal{C}^1$ function and by using (\ref{norm}) and the Poincar\'{e}-Lelong formula (\ref{poinle}), we write (\ref{nonlin}) with our setting,
\begin{equation}\label{samev}Z^{\psi}_{n}(\alpha_{n})=\int_{X}{\big( \log{|s_{n}|_{h^{n}}-\log{\sqrt{K_{n}(x, x)}}} \big)\frac{\sqrt{-1}}{2\pi}\partial \overline{\partial}\phi(x)}= \langle [\text{Div}(s_{n})], \phi \rangle + \zeta_{n, L},\end{equation}where $\zeta_{n, L}= \langle -nc_{1}(L, h) -dd^{c}\log{\sqrt{K_{n}(x, x)}}, \phi \rangle$, namely, $\zeta_{n, L}$ is some constant relying only on the line bundle $L$. Thanks to a basic feature of variance, this last expression (\ref{samev}) tells us that $Z^{\psi}_{n}(\alpha_{n})$ and $\langle [\text{Div}(s_{n})], \phi \rangle$ have the same variances.

In the rest of the proof, we will confirm that the conditions (i) and (ii) of Theorem \ref{sots} do indeed hold in our case. In order to apply Theorem \ref{sots},  we take a compact set $K \subset X$ with $\text{supp}(\phi) \subset K$, which plays the role of $G$ in Theorem \ref{sots}. Since $\lambda(\rho)=\log{\rho}$ is increasing, we only examine the case that $\nu=1$. For both the off-diagonal and near-diagonal asymptotics, we divide the integrals into two parts: $\text{dist}(x, y) \leq b_{0}\sqrt{\frac{\log{n}}{n}}$ and $\text{dist}(x, y) \geq b_{0}\sqrt{\frac{\log{n}}{n}}$, which both integration regions are in $K$. We begin with the condition (ii) which is the easier one. Let $k$ and $b_{0}$ be as in Theorem \ref{asym}. Then we first readily have \begin{equation*}
  \lim_{n\rightarrow \infty} \sup_{x\in K}\int_{\text{dist}(x, y) \geq b_{0}\sqrt{\frac{\log{n}}{n}}}{\Gamma_{n}(x, y)dv_{X}(y)}=\lim_{n\rightarrow \infty} \sup_{x\in K}\int_{\text{dist}(x, y) \geq b_{0}\sqrt{\frac{\log{n}}{n}}}{O(n^{-k})dv_{X}(y)}=0.
\end{equation*}For the integral where $\text{dist}(x, y) \leq b_{0}\sqrt{\frac{\log{n}}{n}}$, due to the relation (\ref{coveb}), $\Gamma_{n}(x, y) \leq 1$. We get \begin{equation*}
                        \lim_{n\rightarrow \infty} \sup_{x\in K}\int_{\text{dist}(x, y) \leq b_{0}\sqrt{\frac{\log{n}}{n}}}{\Gamma_{n}(x, y)dv_{X}(y)} \leq \lim_{n\rightarrow \infty} \sup_{x\in K}\int_{\text{dist}(x, y) \leq b_{0}\sqrt{\frac{\log{n}}{n}}}{1dv_{X}(y)}=0.
                        \end{equation*}

We proceed with the verification of the condition (i). For the integral on the off-diagonal set where $\text{dist}(x, y) \geq b_{0}\sqrt{\frac{\log{n}}{n}}$, when we pass to the limit as $n\rightarrow \infty$, the numerator of the integrand decays to zero faster than the denominator because the corresponding asymptotics \,$O(n^{-k})$\, of the normalized Bergman kernel is of second degree in the numerator.

On the diagonal set where $\text{dist}(x, y) \leq b_{0}\sqrt{\frac{\log{n}}{n}}$, we will introduce an abuse of notation as done in \cite{SZ10} to facilitate the computations: Since $T_{x}X$ is a complex vector space, we can write $\frac{v}{\sqrt{n}}=0 + \frac{v}{\sqrt{n}}$, where $0\in T_{x}X$. Now by the exponential map above, we know $\exp_{x}(0)=x$, so we change the origin $0$ of $T_{x}X$ to $\exp_{x}(0)=x \in X$, and so $\exp_{x}(\frac{v}{\sqrt{n}})=x + \frac{v}{\sqrt{n}}$. Now by Theorem \ref{asym} and (\ref{testf})

\begin{equation}\label{st22}
\liminf_{n\rightarrow \infty}{\frac{\int_{K}dv_{X}\int_{|v|\leq b_{0}\sqrt{\log{n}}}{(1+ o(1))\exp{\big( -\frac{n}{2}\Psi_{x}^{TX}(0, \frac{v}{\sqrt{n}})^{2}\big)}\psi(x)\psi(x+\frac{v}{\sqrt{n}})dv}}{\sup_{x\in K}{\int_{|v|\leq b_{0}\sqrt{\log{n}}}{(1+ o(1))\exp{\big( -\frac{n}{4}\Psi_{x}^{TX}(0, \frac{v}{\sqrt{n}})^{2}\big)}dv}}}}.
\end{equation}

We will bound the integrals in the numerator and the denominator from below by two different integrals by invoking the inequality (\ref{dist1}) between the Riemannian distance on the Hermitian manifold $X$ and the weighted distance function $\Psi$ on the tangent space $T_{x}X$. We will use the left and right parts of the inequality (\ref{dist1}), respectively. For the numerator, from the inequality $D_{1} \text{dist}(\exp_{x}(u), \exp_{x}(u')) \geq \Psi_{x}^{TX}(u, u')$, we have \begin{equation}\label{lp11}
      e^{-\frac{n}{2}\Psi^{TX}_{x}(u, u')^{2}} \geq e^{-\frac{n}{2}D^{2}_{1} \text{dist}(\exp_{x}(u), \exp_{x}(u'))^{2}} .
     \end{equation}Similarly, the inequality $\Psi_{x}^{TX}(u, u') \geq \frac{1}{D_{1}} \text{dist}(\exp_{x}(u), \exp_{x}(u'))$ yields that  \begin{equation}\label{rp22}
       e^{-\frac{n}{4}\Psi^{TX}_{x}(u, u')^{2}} \leq e^{-\frac{n}{4}\frac{1}{D^{2}_{1}} \text{dist}(\exp_{x}(u), \exp_{x}(u'))^{2}}.
     \end{equation} Now, writing $u=0$ and $u'=\frac{v}{\sqrt{n}}$, the numerator of (\ref{st22}) can be bounded from below by (\ref{lp11}),

\begin{align}\label{wert}
\int_{K}dv_{X}\int_{|v|\leq b_{0}\sqrt{\log{n}}}{(1+ o(1))\exp{\big( -\frac{n}{2}\Psi_{x}^{TX}(0, \frac{v}{\sqrt{n}})^{2}\big)}\psi(x)\psi(x+\frac{v}{\sqrt{n}})dv} \\ \nonumber
\quad \geq \int_{K}dv_{X}\int_{|v|\leq b_{0}\sqrt{\log{n}}}{(1+ o(1))e^{-\frac{n}{2}D^{2}_{1} \text{dist}(\exp_{x}(0), \exp_{x}(\frac{v}{\sqrt{n}}))^{2}}\psi(x)\psi(x+\frac{v}{\sqrt{n}})dv} \nonumber &
     \end{align}

As for the denominator in (\ref{st22}), fix $x\in K$ so that the supremum is attained. Then, by using (\ref{rp22}),

\begin{equation}\label{ilg}
  \frac{1}{\int_{|v|\leq b_{0}\sqrt{\log{n}}}{(1+ o(1))\exp{\big( -\frac{n}{4}\Psi_{x}^{TX}(0, \frac{v}{\sqrt{n}})^{2}\big)}dv}} \geq \frac{1}{\int_{|v|\leq b_{0}\sqrt{\log{n}}}{(1+o(1))e^{-\frac{n}{4}\frac{1}{D^{2}_{1}} \text{dist}(\exp_{x}(0), \exp_{x}(\frac{v}{\sqrt{n}}))^{2}}dv}}.
\end{equation}

Recall that $\text{dist}(\exp_{x}(0), \exp_{x}(v))=|v|$. Therefore, when $n\rightarrow \infty$, since $\psi$ is $\mathcal{C}^{1}$, on the right hand sides of (\ref{wert}) and (\ref{ilg}), we obtain two Gaussian integrals (on $\mathbb{C}^{m}$) whose values are $(\frac{2\pi}{D^{2}_{1}})^{m}$ and $(4\pi D^{2}_{1})^{m}$ respectively. Hence, after the cancellations of these two values, the limit in (\ref{st22}) has a
lower bound given by \begin{equation*}\frac{1}{2^{m}D^{4m}_{1}}\int_{X}{\psi^{2}dv_{X}}>0,
                                                                                \end{equation*} which concludes the proof.
\end{proof}

\textbf{Acknowledgments.} Our thanks are extended to the anonymous reviewer whose recommendations have significantly enhanced the presentation of our manuscript.

%%%%%%%%%%%%%%%%%%%%%%%%%%%%%%%%%%%%%%%%%%%%%%%%%%%%%%%%%%%

{}

 \end{document}